\documentclass[12pt]{amsart}

\usepackage{graphicx}
\usepackage{amssymb}
\usepackage{amsthm}
\usepackage{amsmath}
\usepackage{stmaryrd}
\usepackage{cite}
\usepackage[mathscr]{eucal}
\usepackage{hyperref}
\usepackage[margin=1.0in]{geometry}
\usepackage{comment}

\usepackage{url}
\usepackage{hyperref}

\title[Holonomy restrictions]{Holonomy restrictions from the curvature operator of the second kind}
\author[Nienhaus, Petersen, Wink and Wylie]{Jan Nienhaus, Peter Petersen, Matthias Wink and William Wylie}
\address{Department of Mathematics, UCLA, 520 Portola Plaza, Los Angeles, CA, 90095}
\email{petersen@math.ucla.edu}

\address{Mathematisches Institut, Universit\"at M\"unster, Einsteinstra{\ss}e 62, 48149 M\"unster}
\email{j.nienhaus@uni-muenster.de}
\email{mwink@uni-muenster.de}

\address{Department of Mathematics, Syracuse University, 215 Carnegie Building, Syracuse, NY 13244}
\email{wwylie@syr.edu}


\makeatletter
\@namedef{subjclassname@2020}{
\textup{2020} Mathematics Subject Classification}
\makeatother


\thanks{JN acknowledges support by the Alexander von Humboldt Foundation through Gustav Holzegel's Alexander von Humboldt Professorship endowed by the Federal Ministry of Education and Research. JN and MW are funded by the Deutsche Forschungsgemeinschaft (DFG, German Research Foundation) under Germany's Excellence Strategy EXC 2044–390685587, Mathematics M\"unster: Dynamics–Geometry–Structure. WW acknowledges support of NSF grant \#1654034}

\begin{document}
\newcommand{\Ext}{\bigwedge\nolimits}
\newcommand{\Div}{\operatorname{div}}
\newcommand{\Hol} {\operatorname{Hol}}
\newcommand{\diam} {\operatorname{diam}}
\newcommand{\Scal} {\operatorname{Scal}}
\newcommand{\scal} {\operatorname{scal}}
\newcommand{\Ric} {\operatorname{Ric}}
\newcommand{\Hess} {\operatorname{Hess}}
\newcommand{\grad} {\operatorname{grad}}
\newcommand{\Sect} {\operatorname{Sect}}
\newcommand{\Rm} {\operatorname{Rm}}
\newcommand{ \Rmzero } {\mathring{\Rm}}
\newcommand{\Rc} {\operatorname{Rc}}
\newcommand{\Curv} {S_{B}^{2}\left( \mathfrak{so}(n) \right) }
\newcommand{ \tr } {\operatorname{tr}}
\newcommand{ \id } {\operatorname{id}}
\newcommand{ \Riczero } {\mathring{\Ric}}
\newcommand{ \ad } {\operatorname{ad}}
\newcommand{ \Ad } {\operatorname{Ad}}
\newcommand{ \dist } {\operatorname{dist}}
\newcommand{ \rank } {\operatorname{rank}}
\newcommand{\Vol}{\operatorname{Vol}}
\newcommand{\dVol}{\operatorname{dVol}}
\newcommand{ \zitieren }[1]{ \hspace{-3mm} \cite{#1}}
\newcommand{ \pr }{\operatorname{pr}}
\newcommand{\diag}{\operatorname{diag}}
\newcommand{\Lagr}{\mathcal{L}}
\newcommand{\av}{\operatorname{av}}
\newcommand{ \floor }[1]{ \lfloor #1 \rfloor }
\newcommand{ \ceil }[1]{ \lceil #1 \rceil }
\newcommand{\Sym} {\operatorname{Sym}}
\newcommand{\bcirc}{ \ \bar{\circ} \ }

\newtheorem{theorem}{Theorem}[section]
\newtheorem{definition}[theorem]{Definition}
\newtheorem{example}[theorem]{Example}
\newtheorem{remark}[theorem]{Remark}
\newtheorem{lemma}[theorem]{Lemma}
\newtheorem{proposition}[theorem]{Proposition}
\newtheorem{corollary}[theorem]{Corollary}
\newtheorem{assumption}[theorem]{Assumption}
\newtheorem{acknowledgment}[theorem]{Acknowledgment}
\newtheorem{DefAndLemma}[theorem]{Definition and lemma}

\newenvironment{remarkroman}{\begin{remark} \normalfont }{\end{remark}}
\newenvironment{exampleroman}{\begin{example} \normalfont }{\end{example}}
\newenvironment{Beweis}{\begin{proof} \text{} \\}{\end{proof}}

\newcommand{\R}{\mathbb{R}}
\newcommand{\N}{\mathbb{N}}
\newcommand{\Z}{\mathbb{Z}}
\newcommand{\Q}{\mathbb{Q}}
\newcommand{\C}{\mathbb{C}}
\newcommand{\F}{\mathbb{F}}
\newcommand{\X}{\mathcal{X}}
\newcommand{\D}{\mathcal{D}}
\newcommand{\Cont}{\mathcal{C}}

\renewcommand{\labelenumi}{(\alph{enumi})}
\newtheorem{maintheorem}{Theorem}[]
\renewcommand*{\themaintheorem}{\Alph{maintheorem}}
\newtheorem{maincorollary}{Corollary}[]
\newtheorem*{theorem*}{Theorem}
\newtheorem*{corollary*}{Corollary}
\newtheorem*{remark*}{Remark}
\newtheorem*{example*}{Example}
\newtheorem*{question*}{Question}

\begin{abstract}
We show that an $n$-dimensional Riemannian manifold with $n$-nonnegative or $n$-nonpositive curvature operator of the second kind has restricted holonomy $SO(n)$ or is flat. The result does not depend on completeness and can be improved provided the space is Einstein or K\"ahler. In particular, if a locally symmetric space has $n$-nonnegative or $n$-nonpositive curvature operator of the second kind, then it has constant curvature. When the locally symmetric space is irreducible this can be improved to $\frac{3n}{2}\frac{n+2}{n+4}$-nonnegative or $\frac{3n}{2}\frac{n+2}{n+4}$-nonpositive curvature operator of the second kind.
\end{abstract}

\maketitle

\section*{Introduction}

The holonomy group of a Riemannian manifold was introduced by \'E. Cartan, who used it as a tool to classify symmetric spaces \cite{CartanUneClasseRemarquableRiemann}. In particular, the curvature operator of any Riemannian manifold
\begin{align*}
\mathfrak{R} \colon \Lambda^2 TM \to \Lambda^2 TM, \ \left( \mathfrak{R}(\omega) \right)_{ij} = \sum R_{ijkl}  \omega_{kl}
\end{align*}
restricts to the holonomy algebra. Conversely, Ambrose-Singer proved in \cite{AmbroseSingerATheoremOnHolonomy} that the curvature operator determines the holonomy algebra. 

The Gallot-Meyer theorem \cite{GallotMeyerCurvOperatorAndForms} provides a classification of compact Riemannian manifolds with nonnegative curvature operators in terms of their holonomy. Namely, unless the manifold is reducible or locally symmetric, its restricted holonomy is $SO(n)$ or $U(\frac{n}{2})$, and the universal cover is a rational homology sphere or a rational cohomology $\mathbb{CP}^{\frac{n}{2}},$ respectively. In fact, the conclusion on the cohomology can be improved to a diffeomorphism, respectively biholomorphism, classification. This follows from work of Hamilton \cite{Hamilton3DimRF, Hamilton4DimRFposCurvOp}, B\"ohm-Wilking \cite{BW2} and Mok \cite{MokUniformizationKaehler}, see also Brendle-Schoen \cite{BrendleSchoenWeaklyQuarterPinched}. 

In this paper we study restrictions on the holonomy based on curvature conditions for the curvature operator of the second kind, which is the curvature operator on trace-free symmetric $(0,2)$-tensors. More precisely, the self-adjoint operator on symmetric $(0,2)$-tensors
\begin{align*}
\overline{R} \colon S^2(TM) \to S^2(TM), \ \left( \overline{R}(h) \right)_{ij} = \sum R_{iklj} h_{kl}
\end{align*}
induces the {\em curvature operator of the second kind} via
\begin{align*}
\mathcal{R} \colon  S^2_0(TM) \to S^2_0(TM), \ \mathcal{R} = \operatorname{pr}_{S^2_0(TM)} \circ \overline{R}=\overline{R} + g(\Ric, \cdot ) \frac{g}{n}.
\end{align*}

Note that $\mathcal{R}$ is called $k$-nonnegative provided its eigenvalues $\lambda_1 \leq \lambda_2 \leq \ldots \leq \lambda_N$ satisfy $\lambda_1 + \ldots + \lambda_{\floor{k}} + \left( k - \floor{k} \right) \lambda_{\floor{k}+1} \geq 0.$ $\mathcal{R}$ is nonnegative if it is $1$-nonnegative. \vspace{2mm}

In \cite{NishikawaDeformationRiemMetrics} Nishikawa conjectured that a compact manifold with positive curvature operator of the second kind is diffeomorphic to a space form, and in case the curvature operator of the second kind is nonnegative, the manifold is diffeomorphic to a locally symmetric space.

Recently Cao-Gursky-Tran \cite{CaoGurskyTranNishikawaConjecture} proved that indeed compact manifolds with $2$-positive curvature operators of the second kind are diffeomorphic to space forms. X.Li \cite{LiCurvatureOperatorSecondKind} generalized the result to manifolds with $3$-positive curvature operators of the second kind. Both proofs rely on the observation that the manifolds must satisfy the $PIC_1$ condition, and Brendle's convergence theorem \cite{BrendleConvergenceInHihgerDimensions} for the Ricci flow applies. In fact, in \cite{LiFourPointFiveNonnegCurvII}, X.Li proved that manifolds with $4 \frac{1}{2}$-positive curvature operators of the second kind have positive isotropic curvature. 

In the rigidity case, the authors \cite{NienhausPetersenWinkBettiNumbersCOSK} proved that compact manifolds with $3$-nonnegative curvature operators of the second kind are either flat or diffeomorphic to a spherical space form, eliminating compact symmetric spaces from Nishikawa's conjecture. This is a consequence of X.Li's work \cite{LiCurvatureOperatorSecondKind} and the fact that compact manifolds with $\frac{n+2}{2}$-nonnegative curvature operators of the second kind are either flat or rational homology spheres, \cite{NienhausPetersenWinkBettiNumbersCOSK}. 

\vspace{2mm}

The aim of this paper is to consider the implications of the Bochner formulas in \cite{NienhausPetersenWinkBettiNumbersCOSK} to the local geometry of Riemannian manifolds. We achieve this by replacing the global topological restrictions of \cite{NienhausPetersenWinkBettiNumbersCOSK} for compact Riemannian manifolds with restrictions on holonomy. 

For example, it is natural to ask whether manifolds with $k$-nonnegative curvature operators of the second kind are flat or have restricted holonomy $SO(n)$. Notice that this question is also valid for manifolds with $k$-nonpositive curvature operators of the second kind, where by definition $\mathcal{R}$ is $k$-nonpositive if $-\mathcal{R}$ is $k$-nonnegative. Moreover, these are local questions and similarly apply to possibly incomplete manifolds. 

\vspace{2mm}

We note that all manifolds are assumed to be connected. 

\begin{maintheorem} 
\label{MainTheoremSOorTrivial}
Let $(M,g)$ be an $n$-dimensional, not necessarily complete Riemannian manifold. If the curvature operator of the second kind is $n$-nonnegative or $n$-nonpositive, then the restricted holonomy of $(M,g)$ is $SO(n)$ or $(M,g)$ is flat. 

In particular, if $(M,g)$ is in addition a locally symmetric space, then it has constant curvature. 
\end{maintheorem} 

The example of $S^{n-1} \times S^1$ shows that the result cannot be improved to $(n+1)$-nonnegative curvature operator of the second kind, cf. \cite[Example 2.6]{LiCurvatureOperatorSecondKind} or example \ref{CurvatureSnS1}.

Previously, X.Li proved that complete manifolds with $n$-nonnegative curvature operators of the second kind are irreducible, \cite[Theorem 1.8]{LiCurvatureOperatorSecondKind}. Theorem \ref{MainTheoremSOorTrivial} strengthens this result by allowing incomplete manifolds and by excluding non-flat symmetric spaces with non-generic holonomy. In fact, for Einstein manifolds, e.g., irreducible symmetric spaces or quaternion K\"ahler manifolds, we have the following improvement.

\begin{maintheorem} 
\label{MainTheoremEinstein}
Let $n \geq 3$ and let $(M,g)$ be an $n$-dimensional Einstein manifold. 

If the curvature operator of the second kind is $N$-nonnegative or $N$-nonpositive for some $N<\frac{3n}{2}\frac{n+2}{n+4}$, then the restricted holonomy of $(M,g)$ is $SO(n)$ or $(M,g)$ is flat. 
\end{maintheorem} 

The theory of Diophantine equations implies that $\frac{3n}{2}\frac{n+2}{n+4}$ is only an integer for $n=0,2,8.$ In particular, in all other dimensions $\floor{\frac{3n}{2}\frac{n+2}{n+4}}$-nonnegativity or $\floor{\frac{3n}{2}\frac{n+2}{n+4}}$-nonpositivity of the curvature operator of the second kind implies that the restricted holonomy of $(M,g)$ is $SO(n)$ or $(M,g)$ is flat. \vspace{2mm}

Theorems \ref{MainTheoremSOorTrivial} and \ref{MainTheoremEinstein} together give many new examples of spaces that do not have $N$-nonnegative and $N$-nonpositive curvature operator of the second kind. In particular, as a corollary of the proof of Theorem \ref{MainTheoremEinstein} we obtain the following statement for irreducible locally symmetric spaces.

\begin{corollary*}
Let $(M,g)$ be an $n$-dimensional irreducible locally symmetric space.

If the curvature operator of the second kind is $N$-nonnegative or $N$-nonpositive for some $N< \frac{3n}{2}\frac{n+2}{n+4}$, then $(M,g)$ has constant curvature.
\end{corollary*}

According to Berger's classification of holonomy groups, unless $(M,g)$ is reducible or Einstein, it has restricted holonomy $SO(n)$ or $U(m)$. In the K\"ahler case, X.Li proved that possibly incomplete K\"ahler manifolds with $4$-nonnegative curvature operators of the second kind are flat, \cite[Theorem 1.9]{LiCurvatureOperatorSecondKind}, see also \cite{LiKaehlerSurfaces} for an improvement in the case of K\"ahler surfaces. With our methods we can relax the assumptions to a nonnegativity condition which depends linearly on the dimension of $M.$ In addition, we can also include the corresponding nonpositivity condition.

\begin{maintheorem} 
\label{MainTheoremKaehler}
Let $(M,g)$ be a K\"ahler manifold of real dimension $2m.$ Set 
\begin{align*}
C^{\text{K\"ahler}}(m)=\begin{cases}
3m\frac{m+1}{m+2} & \text{if } m \text{ even,} \\
3m\frac{(m+1)(m^{2}-1)}{(m+2)(m^{2}+1)} & \text{if } m \text{ odd.} \\
\end{cases}
\end{align*}
If the curvature operator of the second kind is $C'$-nonnegative or $C'$-nonpositive for some $C'<C^{\text{K\"ahler}},$ then $(M,g)$ is flat.
\end{maintheorem}
 
Note that if $m$ is even, then $C^{\text{K\"ahler}}(m)=N(2m),$ where $N=\frac{3n}{2}\frac{n+2}{n+4}$, as in the Einstein case.

Subsequently X.Li could improve the assumptions in Theorem \ref{MainTheoremKaehler} to $C^{\text{K\"ahler}}(m) \sim m^2$ in \cite{LiKaehler} using different methods. \vspace{2mm}

The proofs of the main theorems are based on the Bochner formulas developed in \cite{NienhausPetersenWinkBettiNumbersCOSK}. If $\omega$ is a harmonic $p$-form, then it satisfies the Bochner formula 
\begin{align*}
\Delta \frac{1}{2} |\omega|^{2} = |\nabla\omega|^{2}+g(\Ric_{L}(\omega),\omega),
\end{align*}
where $Ric_L$ is the Lichnerowicz-Laplacian, cf. \cite[Chapter 9]{PetersenRiemGeom}.  
Due to \cite[Proposition 2.1]{NienhausPetersenWinkBettiNumbersCOSK}, this curvature term satisfies
\begin{align*}
\frac{3}{2}g(\Ric_{L}(\omega),\omega)=\sum_{\alpha=1}^{N}\lambda_{\alpha}|S_{\alpha}\omega|^{2}+\frac{p(n-2p)}{n}\sum_{j}g\left(i_{\Ric\left(e_{j}\right)}\omega,i_{e_{j}}\omega\right)+\frac{p^{2}}{n^{2}}\scal|\omega|^{2},
\end{align*}
where $\lbrace S_{\alpha}\rbrace$ is an orthonormal eigenbasis of the curvature operator of the second kind with corresponding eigenvalues $\lbrace\lambda_{\alpha}\rbrace$, $N=\dim S_{0}^{2}(TM)=\frac{1}{2}(n-1)(n+2),$ and for $S \in S_0^2(TM)$ we have $(S \omega)(X_1, \ldots, X_p) = \sum_{i=1}^p \omega( X_1, \ldots, SX_i, \ldots, X_p).$

The classical strategy of the Bochner technique as carried out in \cite{NienhausPetersenWinkBettiNumbersCOSK} is to estimate the curvature term $g(\Ric_{L}(\omega),\omega)$, and to conclude via the maximum principle that harmonic forms on compact manifolds are parallel provided $g(\Ric_{L}(\omega),\omega) \geq 0.$ Notice that this argument does not apply analogously to nonpositivity conditions. \vspace{2mm}

The key idea of this paper is that unless the restricted holonomy is generic, there exists a parallel form, at least locally on the manifold. A subtle point is that in dimension $n=5$ there is one exception, namely the holonomy representation associated to the pair of symmetric spaces $SU(3)/SO(3)$ and $SL(3,\mathbb{C})/SO(3).$ 

This relies on Berger's classification \cite{BergerHolonomyClassfication} of holonomy groups. If $(M,g)$ locally splits as a product, then the volume form of one of the factors induces a local parallel form. In the irreducible case, the reduction of the holonomy group to a holonomy group in Berger's list other than $SO(n)$ implies the existence of a local parallel form, cf. \cite[Section 10.109]{BesseEinstein}. Finally, compact symmetric spaces which are rational homology spheres are classified by Wolf in \cite{WolfSymmetricRealCohomSpheres}. This leads to the exception of $SU(3)/SO(3)$ as it is a simply connected rational homology sphere with $H_2(SU(3)/SO(3),\Z)=\Z / 2 \Z.$ \vspace{2mm}

However, any locally defined parallel form satisfies the equation
\begin{align*}
g(\Ric_{L}(\omega),\omega) = 0
\end{align*}
on some open set. 

The estimates in \cite{NienhausPetersenWinkBettiNumbersCOSK} imply that, under the curvature assumptions in Theorems \ref{MainTheoremSOorTrivial} - \ref{MainTheoremKaehler}, no local parallel form exists unless the manifold is flat. Consequently, the manifold has restricted holonomy $SO(n)$ or it is flat. The holonomy representation of the symmetric space $SU(3)/SO(3)$ does not occur as an exception as $SU(3)/SO(3)$ violates the curvature assumption.

As this argument relies on the equation $g(\Ric_{L}(\omega),\omega) = 0,$ it carries over to nonpositivity conditions on the curvature operator of the second kind. 

\vspace{2mm}

\textit{Structure.} In section \ref{SectionNonexistenceParallelForms} we prove nonexistence results for parallel forms provided the curvature operator is sufficiently nonnegative or nonpositive, respectively. This is based on the results in \cite{NienhausPetersenWinkBettiNumbersCOSK}. In section \ref{SectionIrreducibility} we generalize X.Li's result on irreducibility \cite[Theorem 1.8]{LiCurvatureOperatorSecondKind} to possibly incomplete Riemannian manifolds with $n$-nonnegative or $n$-nonpositive curvature operators of the second kind. We also obtain an improvement for Einstein manifolds. Section \ref{SectionProofs} combines the results from the previous sections with the existence of local parallel forms for manifolds whose restricted holonomy is not $SO(n)$ or associated to the pair of symmetric spaces $SU(3)/SO(3)$ or $SL(3,\mathbb{C})$, cf. lemma \ref{HolonomyAndLocalForms}, to prove Theorems \ref{MainTheoremSOorTrivial} - \ref{MainTheoremKaehler}. \vspace{2mm}

The reader is referred to \cite[Chapter 10]{BesseEinstein} for an introduction to holonomy groups. \vspace{2mm}

\textit{Remark.} Shortly after this paper first appeared in pre-print form, we learned from Xiaolong Li that he had independently obtained an improvement of Theorem \ref{MainTheoremKaehler}, see \cite[Theorem 1.2]{LiKaehler}. Subsequently, Li was also able to combine our work with his own to improve Theorem \ref{MainTheoremSOorTrivial} to manifolds with $\left(n + \frac{n-2}{n}\right)$-nonnegative or $\left(n + \frac{n-2}{n}\right)$-nonpositive curvature operators of the second kind, see \cite[Theorem 1.3]{LiProducts}.

\section{Nonexistence of parallel forms}
\label{SectionNonexistenceParallelForms}

In this section we rephrase \cite[Theorem B and Theorem C]{NienhausPetersenWinkBettiNumbersCOSK} and obtain results on the nonexistence of parallel forms based on curvature conditions for the curvature operator of the second kind. Moreover, the reformulation also allows us to include nonpositivity conditions on the eigenvalues of the curvature operator of the second kind. This is an extension of \cite[Theorem B and Theorem C]{NienhausPetersenWinkBettiNumbersCOSK} where only nonnegativity conditions are considered. In addition, the results are entirely local, and completeness of the metric is not required. \vspace{2mm}

Recall that the curvature operator of the second kind $\mathcal{R}$ is called $k$-nonnegative if its eigenvalues $\lambda_1 \leq \lambda_2 \leq \ldots \leq \lambda_N$ satisfy $\lambda_1 + \ldots + \lambda_{\floor{k}} + \left( k - \floor{k} \right) \lambda_{\floor{k}+1} \geq 0.$ Moreover, $\mathcal{R}$ is called $k$-nonpositive if $-\mathcal{R}$ is $k$-nonnegative.

\begin{remark}
\label{NonnegAndScalFlatIsFlat}
\normalfont
Suppose that the curvature operator of the second kind $\mathcal{R}$ is $k$-nonnegative or $k$-nonpositive for some $k<\dim S_0^2(TM).$ If the scalar curvature vanishes, then $\mathcal{R}=0.$ This is an immediate consequence of the fact that $\tr ( \mathcal{R} ) = \frac{n+2}{2n} \scal.$
\end{remark}

\begin{lemma} 
\label{NoParallelFormsEinstein}
Let $(M,g)$ be an $n$-dimensional Einstein manifold. Let $N=\frac{3n}{2}\frac{n+2}{n+4}$ and let $\omega$ be a non-vanishing parallel $p$-form with $1 \leq p\leq n/2$.

If the curvature operator of the second kind is $N'$-nonnegative or $N'$-nonpositive for some $N'<N,$ then $(M,g)$ is flat.
\end{lemma}
\begin{proof}
As explained in the introduction, by the Bochner technique, the existence of a parallel $p$-form $\omega$ for some $1 \leq p\leq \frac{n}{2}$ implies that we have
\begin{align*}
    g( \Ric_L(\omega), \omega) = 0.
\end{align*}

If $g$ is Einstein, by \cite[Proposition 2.1]{NienhausPetersenWinkBettiNumbersCOSK}, we have the explicit formula
\begin{align*}
0 = \frac{3}{2} g( \Ric_L(\omega), \omega) = \sum_{\alpha} \lambda_{\alpha}|S_{\alpha}\omega|^{2} + \frac{p(n-p)}{n^2} \scal | \omega |^2,
\end{align*}
where $\lbrace S_{\alpha}\rbrace$ is an orthonormal eigenbasis of the curvature operator of the second kind with corresponding eigenvalues $\lbrace\lambda_{\alpha}\rbrace$. Furthermore, \cite[Proposition 3.16]{NienhausPetersenWinkBettiNumbersCOSK} says that 
\begin{align*}
\sum_{\alpha} \lambda_{\alpha}|S_{\alpha}\omega|^{2} + \frac{p(n-p)}{n^2} \scal | \omega |^2 \geq \frac{p(n-p)}{n} \left[ \mathcal{R}, \frac{n+4}{n+2}, \frac{3n}{2} \right] \cdot | \omega |^2,
\end{align*}
where $\left[ \mathcal{R}, \frac{n+4}{n+2}, \frac{3n}{2} \right]$ denotes a finite weighted sum $\sum_{\alpha} \omega_{\alpha} \lambda_{\alpha}$ in terms of the eigenvalues $\lambda_{\alpha}$ with highest weight $\max_{\alpha} \omega_{\alpha} = \frac{n+4}{n+2}$ and total weight $\sum_{\alpha} \omega_{\alpha} = \frac{3n}{2},$ see the introductory discussion of \cite[Section 3]{NienhausPetersenWinkBettiNumbersCOSK} for further background.

Suppose that $\mathcal{R}$ is $N'$-nonnegative for some $N'<N.$ The weight principle \cite[Theorem 3.6 (d)]{NienhausPetersenWinkBettiNumbersCOSK} implies that $\mathcal{R}$ is either $N$-positive or $1$-nonnegative. If $\mathcal{R}$ is $N$-positive, then 
\begin{align*}
0 \geq \frac{p(n-p)}{n} \left[ \mathcal{R}, \frac{n+4}{n+2}, \frac{3n}{2} \right] \cdot | \omega |^2
\end{align*}
implies that $\omega$ vanishes. Otherwise $(M,g)$ is flat or a rational homology sphere by \cite[Theorem A]{NienhausPetersenWinkBettiNumbersCOSK}. 

When $\mathcal{R}$ is $N'$-nonpositive, then the above argument applied to $-\mathcal{R}$ yields the claim. 
\end{proof}

Let
\begin{align*}
C_{p}=C_{p}(n)=\frac{3}{2}\frac{n(n+2)p(n-p)}{n^{2}p-np^{2}-2np+2n^{2}+2n-4p}.
\end{align*}

\begin{lemma} 
\label{NoParallelForms}
Let $(M,g)$ be an $n$-dimensional Riemannian manifold and let $\omega$ be a non-vanishing parallel $p$-form with $1 \leq p\leq n/2$. 

If the curvature operator of the second kind is $C'$-nonnegative or $C'$-nonpositive for some $C'<C_p,$ then $(M,g)$ is flat.
\end{lemma}
\begin{proof}
Any parallel $p$-form $\omega$ satisfies
\begin{align*}
0 = \frac{3}{2} g( \Ric_L(\omega), \omega) = \sum_{\alpha}\lambda_{\alpha}|S_{\alpha}\omega|^{2}+\frac{p(n-2p)}{n}\sum_{j}g\left(i_{\Ric\left(e_{j}\right)}\omega,i_{e_{j}}\omega\right)+\frac{p^{2}}{n^{2}}\scal|\omega|^{2},
\end{align*}
according to \cite[Proposition 2.1]{NienhausPetersenWinkBettiNumbersCOSK}. One can now proceed as in the proof of lemma \ref{NoParallelFormsEinstein}. Instead of \cite[Proposition 3.16]{NienhausPetersenWinkBettiNumbersCOSK} one uses \cite[Proposition 3.14]{NienhausPetersenWinkBettiNumbersCOSK} to estimate the curvature term.
\end{proof}

\section{Irreducibility}
\label{SectionIrreducibility}

In this section we offer a different proof of X.Li's result \cite[Theorem 1.8]{LiCurvatureOperatorSecondKind} about irreducibility of manifolds with $n$-nonnegative curvature operators of the second kind. Our technique also generalizes to manifolds with $n$-nonpositive curvature operators of the second kind and allows for an improvement in the case of Einstein manifolds. In addition, our method is entirely local and completeness of the metric is not required. It is based on a Bochner formula and the following basic observation:

\begin{lemma}
\label{KernelCOSKofProducts}
The kernel of the curvature operator of the second kind of a Riemannian product $\left(M_{1}^{p}\times M_{2}^{n-p},g_{1}\oplus g_{2}\right)$ is at least $p(n-p)$-dimensional.
\end{lemma}
\begin{proof}
Select an orthonormal basis $e_{1},...,e_{p},e_{p+1},...,e_{n}$ where the first $p$ vectors are tangent to $M_{1}$. Since the tangent distributions to each factor are parallel we note that $R_{ijkl}=0$ provided $j\leq p$ and $k\geq p+1$. Since $\overline{R}( e^i \otimes e^j + e^j \otimes e^i ) = R_{\cdot ij \cdot} + R_{\cdot ji \cdot}$ it follows that
\begin{align*}
\mathcal{R}(  e^i \otimes e^j + e^j \otimes e^i ) = \overline{R}( e^i \otimes e^j + e^j \otimes e^i ) = 0
\end{align*}
for all $1 \leq i \leq p$ and $p+1 \leq j \leq n.$  
\end{proof}

\begin{proposition} 
\label{IrreducibleOrFlat}
Let $(M,g)$ be an $n$-dimensional, not necessarily complete, Riemannian manifold. If the curvature operator of the second kind $\mathcal{R}$ is $n$-nonnegative or $n$-nonpositive, then $\left(M,g\right)$ is irreducible or flat.
\end{proposition} 
\begin{proof}
Suppose that $(M,g)$ splits locally as a Riemannian product $\left(M_{1}^{p}\times M_{2}^{n-p},g_{1}\oplus g_{2}\right)$ with $1 \leq p\leq n-p$.

Due to lemma \ref{KernelCOSKofProducts}, $\mathcal{R}$ has a kernel of dimension $\dim \operatorname{ker} \mathcal{R} \geq p\left(n-p\right)\geq n-1$. In particular, the assumption that $\mathcal{R}$ is $n$-nonnegative or $n$-nonpositive forces $\mathcal{R}$ to be nonnegative or nonpositive.

Let $\omega$ denote the pullback of the volume form of $\left(M_{1},g_{1}\right)$ to an open set $U \subset M.$ Note that $\omega$ is a nonvanishing parallel form on $U.$ Thus it satisfies the Bochner formula
\begin{align*}
0=\frac{3}{2} g(\Ric_L(\omega), \omega) = \sum_{\alpha}\lambda_{\alpha}|S_{\alpha}\omega|^{2}+\frac{p(n-2p)}{n}\sum_{j}g\left(i_{\Ric\left(e_{j}\right)}\omega,i_{e_{j}}\omega\right)+\frac{p^{2}}{n^{2}}\scal|\omega|^{2},
\end{align*}
cf. \cite[Proposition 2.1]{NienhausPetersenWinkBettiNumbersCOSK}. Since the curvature operator of the second kind is nonnegative or nonpositive, all curvature terms have the same sign, so the sum can only vanish if $\scal = 0$. Thus, $\tr \left( \mathcal{R} \right) = \frac{n+2}{2n} \scal = 0$ and $\mathcal{R}$ also vanishes identically, i.e., $(U,g_{|U})$ is flat. Since the restricted holonomy does not depend on $q \in M,$ $(M,g)$ is either irreducible or flat. 
\end{proof}

\begin{remark}
\normalfont
In \cite[Proposition 5.1]{LiCurvatureOperatorSecondKind}, X.Li proved that if $(M_1^p \times M_2^{n-p}, g_1 \oplus g_2)$ has $(p(n-p)+1)$-nonnegative curvature operator of the second kind, then it is flat. 

This also applies to product manifolds with $(p(n-p)+1)$-nonpositive curvature operators of the second kind. Indeed, according to Lemma \ref{KernelCOSKofProducts}, a product manifold with $(p(n-p)+1)$-nonnegative (or $(p(n-p)+1)$-nonpositive) curvature operator of the second kind in fact has nonnegative (or nonpositive) curvature operator of the second kind. The argument in the proof of proposition \ref{IrreducibleOrFlat} again implies that the manifold is flat. 

The example of $S^p \times S^{n-p}$ below shows that this result is optimal. 
\end{remark}

\begin{example}
\normalfont
\label{CurvatureSnS1}
The curvature operator of the second kind of $S^p \times S^{n-p}$. The case $p=1$ is discussed in \cite[Example 2.6]{LiCurvatureOperatorSecondKind}. Let $e_1, \ldots, e_p$ denote orthonormal tangent vectors corresponding to $S^p$ and let $e_{p+1}, \ldots, e_n$ denote orthonormal tangent vectors corresponding to $S^{n-p}.$ Let 
\begin{align*}
E^{ij} = 
\begin{cases}
\frac{1}{\sqrt{2}} (e^i \otimes e^j + e^j \otimes e^i) & \ \text{ if } \ i \neq j, \\
e^i \otimes e^i & \ \text{ if } \ i = j.
\end{cases}
\end{align*}
The eigenspaces for the curvature operator of the second kind $\mathcal{R}= \mathcal{R}_{S^p \times S^{n-p}}$ are given by 
\begin{align*}
\ker \mathcal{R} = & \ \operatorname{span} \left\lbrace E^{ij} \ \vert \ 1 \leq i \leq p < j \leq n \right\rbrace, \\
\operatorname{Eig}_1(\mathcal{R}) = & \ \operatorname{span} \left\lbrace E^{ij} \ \vert \ 1 \leq i < j \leq p \ \text{ or } \ p+1 \leq i < j \leq n \right\rbrace \\
& \ \oplus \operatorname{span} \left\lbrace E^{11} - E^{ii} \ \vert \ 1 < i \leq p  \right\rbrace \\
& \ \oplus \operatorname{span} \left\lbrace E^{ii} - E^{nn} \ \vert \ p+1 \leq i < n  \right\rbrace,  \\
\operatorname{Eig}_{1- 2 \frac{p(n-p)}{n}}(\mathcal{R}) = & \ \operatorname{span} \left\lbrace \frac{1}{p} \left( E^{11} + \ldots + E^{pp} \right) - \frac{1}{n-p} \left( E^{p+1,p+1} + \ldots + E^{nn} \right) \right\rbrace.
\end{align*}
Note that for $n \geq 3$ the eigenvalue $1- 2 \frac{p(n-p)}{n}$ is negative.
\end{example}

In the Einstein case we have the following improvement of proposition \ref{IrreducibleOrFlat}:

\begin{proposition}
\label{IrreducibilityEinsteinCase}
Let $(M,g)$ be an $n$-dimensional Einstein manifold, and set $N=\frac{3n}{2}\frac{n+2}{n+4}$. If the curvature operator of the second kind is $N'$-nonnegative or $N'$-nonpositive for some $N'<N,$ then $\left(M,g\right)$ is irreducible or flat.
\end{proposition}
\begin{proof}
Suppose that $\left(M^{n},g\right)$ splits locally as $\left(M_{1}^{p}\times M_{2}^{n-p},g_{1}\oplus g_{2}\right)$ with $1 \leq p \leq n-p.$ Let $\omega$ denote the pullback of the volume form of $M_1$ to an open set $U \subset M.$ Note that $\omega$ is a nonvanishing parallel $p$-form on $U.$ According to lemma \ref{NoParallelFormsEinstein}, $(U,g_{|U})$, and hence $(M,g)$, is flat.
\end{proof}

\section{Proofs of the main Theorems}
\label{SectionProofs}

In this section we prove the theorems from the introduction. We start with the K\"ahler case. \vspace{2mm}

\textit{Proof of Theorem \ref{MainTheoremKaehler}}. For a K\"ahler metric the K\"ahler form $\omega$ is a parallel 2-form and its powers $\omega^{p}$ are parallel $2p$-forms. When $m$ is even we use $p=m/2$ and when $m$ is odd $p=(m-1)/2$ to obtain a nontrivial parallel form of degree $m$ or $m-1$, respectively. 

Thus we can apply lemma \ref{NoParallelForms} with 
\begin{align*}
C_m =  3m\frac{m+1}{m+2}
\end{align*}
when $m$ is even and
\begin{align*}
C_{m-1} = 3m \frac{(m+1)(m^2-1)}{(m+2)(m^2+1)}
\end{align*}
when $m$ is odd.  $\hfill \Box$ 

\vspace{2mm}

Note that if $(M,g)$ is not locally symmetric with irreducible holonomy representation, then its restricted holonomy is contained in Berger's list of holonomy groups \cite{BergerHolonomyClassfication}. In particular, unless the restricted holonomy is $SO(n),$ $(M,g)$ admits a non-vanishing parallel form, cf. \cite[Section 10.109]{BesseEinstein}.

For the proofs of Theorems \ref{MainTheoremSOorTrivial} and Theorem \ref{MainTheoremEinstein} we first observe that apart from  $SO(n)$ there is only one more holonomy representation that prevents the existence of a local parallel form. 

\begin{lemma}
\label{HolonomyAndLocalForms}
Unless the restricted holonomy representation of an $n$-dimensional, not necessarily complete Riemannian manifold $(M,g)$ is given by
\begin{enumerate}
\item the standard representation of $SO(n)$ on $\R^n$ or
\item the representation of $SO\left(3\right)$ on $\mathbb{R}^{5}$ associated to the pair of symmetric spaces $SU(3)/SO(3)$ and $SL(3,\mathbb{C})/SO(3)$ for $n=5,$
\end{enumerate}
$(M,g)$ admits a local parallel form for some $0<p \leq \frac{n}{2}$. That is, for every point in $M$ there is an open neighborhood $U$ together with $p$-parallel form defined on $U$ for some $0<p \leq \frac{n}{2}$.
\end{lemma}

\begin{proof}
In case the holonomy is reducible we obtain a local product structure. Thus we can assume $\left(M^{n},g\right)=\left(M_{1}^{p}\times M_{2}^{n-p},g_{1}\oplus g_{2}\right)$ with $p \leq n-p.$ In this case the volume form on the first factor pulls back to a parallel $p$-form on $M$. 

When the holonomy is transitive on the unit sphere in $T_{p}M$, the reduction of the holonomy induces a local parallel form. The specific forms are mentioned in \cite[Sections 10.109-10.111]{BesseEinstein}. 

Thus we may assume that the restricted holonomy is irreducible and does not act transitively on the unit sphere. The theorem of Berger \cite{BergerHolonomyClassfication} and Simons \cite{SimonsTransitivityHolonomy}, see also \cite[Theorem 10.90]{BesseEinstein}, implies that $M$ is locally symmetric. In particular, it corresponds to a unique irreducible simply connected symmetric space.

Irreducible simply connected symmetric spaces come in pairs of a compact and a noncompact symmetric space with the same holonomy representation, see \cite[Chapters 7, 10]{BesseEinstein}. For such a pair with the same holonomy representation, the holonomy principle implies that if one has a parallel form so will the other. This allows us to restrict attention to compact symmetric spaces. 

In the compact case, due to a result of \'E. Cartan,  parallel forms generate the real cohomology and in fact every cohomology class contains a parallel form, see \cite{WolfSpacesOfConstantCurvature}. Thus we are finally reduced to the question of which compact simply connected irreducible symmetric spaces have the de Rham cohomology of spheres. This question is answered for all compact symmetric spaces by Wolf in \cite{WolfSymmetricRealCohomSpheres} and shows that in the simply connected case only round spheres and $SU(3)/SO(3)$ are rational homology spheres.
\end{proof}

\begin{corollary}
\label{ParallelFormsSymmetricSpaces}
Let $(M,g)$ be a simply connected symmetric space. Unless $(M,g)$ is the round sphere, hyperbolic space, $SU(3)/SO(3)$ or $SL(3,\mathbb{C})/SO(3)$, $(M,g)$ admits a parallel $p$-form for some $0<p \leq \frac{1}{2} \dim M.$
\end{corollary}

\textit{Proof of Theorem \ref{MainTheoremEinstein}}. By proposition \ref{IrreducibilityEinsteinCase} we can assume that $(M,g)$ is irreducible. 

Suppose in addition that $(M,g)$ is locally symmetric. According to \cite[Example 4.5]{NienhausPetersenWinkBettiNumbersCOSK}, if $(M,g)$ is locally modeled on $SU(3)/SO(3)$, then its curvature operator of the second kind is $9$-positive but not $8$-nonnegative. However, Theorem \ref{MainTheoremEinstein} requires $\frac{35}{6}$-nonnegativity in dimension $n=5,$ so this case does not arise. Similarly, $(M,g)$ cannot be locally modeled on the dual symmetric space $SL\left(3,\mathbb{C}\right)/SO\left(3\right)$. Thus, unless $(M,g)$ has restricted holonomy $SO(n)$, it admits a local parallel $p$-form for $0<p \leq \frac{1}{2} \dim M$ according to corollary \ref{ParallelFormsSymmetricSpaces}. However, the existence of local parallel $p$-forms is excluded by lemma \ref{NoParallelFormsEinstein} and therefore this case does not occur either. Overall, Theorem \ref{MainTheoremEinstein} follows if $(M,g)$ is locally symmetric. 

Recall that according to Berger's classification of holonomy groups, if $(M,g)$ is neither locally symmetric nor has restricted holonomy $SO(n)$, then $(M,g)$ has reduced holonomy. 

Unless $(M,g)$ is K\"ahler or quaternion K\"ahler, $(M,g)$ is Ricci flat. Due to the assumption on the eigenvalues of the curvature operator of the second kind, remark \ref{NonnegAndScalFlatIsFlat} implies that $(M,g)$ is flat. If $(M,g)$ is K\"ahler, then the K\"ahler form is a non-vanishing parallel $2$-form on $M.$ If $(M,g)$ is quaternion K\"ahler, then the Kraines form is a nonvanishing parallel $4$-form on $M.$ However, due to lemma \ref{NoParallelFormsEinstein}, this is impossible unless $(M,g)$ is flat. 

Consequently, $(M,g)$ has restricted holonomy $SO(n)$ or is flat. $\hfill \Box$

\vspace{2mm}

\textit{Proof of Theorem \ref{MainTheoremSOorTrivial}}. Due to proposition \ref{IrreducibleOrFlat} we may assume that $(M,g)$ is irreducible. If $(M,g)$ is in addition locally symmetric, then it is Einstein, and the claim follows from Theorem \ref{MainTheoremEinstein}. We again apply Berger's holonomy classification to see that either the restricted holonomy is $SO(n)$ or reduced. The case of reduced holonomy is treated as in the proof of Theorem \ref{MainTheoremEinstein}, with lemma \ref{NoParallelFormsEinstein} replaced by lemma \ref{NoParallelForms}. $\hfill \Box$


\end{document}